\documentclass[12pt]{article}
\usepackage{amsmath}
\usepackage{amssymb}

\newtheorem{theorem}{Theorem}
\newtheorem{proposition}{Proposition}

\newtheorem{lemma}{Lemma}
\newtheorem{remark}{Remark}
\newenvironment{proof}[1][Proof]{\noindent\textbf{#1.} }{\ \rule{0.5em}{0.5em}}
\input{tcilatex}

\begin{document}

\title{Midpoint Diagonal Quadrilaterals}
\author{Alan Horwitz}
\date{2/17/21}
\maketitle

\begin{abstract}
A convex quadrilateral, \$Q\$, is called a midpoint diagonal quadrilateral
if the intersection point of the diagonals of \$Q\$ coincides with the
midpoint of at least one of the diagonals of \$Q\$. A parallelogram, P, is a
special case of a midpoint diagonal quadrilateral since the diagonals of P\
bisect one another. We prove two results about ellipses inscribed in
midpoint diagonal quadrilaterals, which generalize properties of ellipses
inscribed in parallelograms involving convex quadrilaterals. First, \$Q\$ is
a midpoint diagonal quadrilateral if and only if each ellipse inscribed in
\$Q\$\ has tangency chords which are parallel to one of the diagonals of
\$Q\$. Second, \$Q\$ is a midpoint diagonal quadrilateral if and only if
each ellipse inscribed in \$Q\$ has a unique pair of conjugate diameters
parallel to the diagonals of \$Q\$. Finally, we show that there is a unique
ellipse, \$E\_I\$, of minimal eccentricity inscribed in a midpoint diagonal
quadrilateral, \$Q\$, and also that the unique pair of conjugate diameters
parallel to the diagonals of \$Q\$ are the \textbf{equal} conjugate
diameters of \$E\_I\$.
\end{abstract}

\textbf{Introduction}

Given a diameter, $l$, of an ellipse, $E_{0}$, there is a unique diameter, $%
m $, of $E_{0}$ such that the midpoints of all chords parallel to $l$ lie on 
$m $. In this case we say that $l$ and $m$ are conjugate diameters of $E_{0}$%
, or that $m$ is a diameter of $E$ conjugate to $l$. $l$ and $m$ are called 
\textbf{equal} conjugate diameters if $\left\vert l\right\vert =\left\vert
m\right\vert $.

We say that $E_{0}$ is \textit{inscribed}\textbf{\ }in a convex
quadrilateral, $Q$, if $E_{0}$ lies inside $Q$ and is tangent to each side
of $Q$. A tangency chord is any chord connecting two points where $E_{0}$ is
tangent to two different sides of $Q$. There are two interesting
properties(probably mostly known) of ellipses inscribed in parallelograms,
P, which involve tangency chords and conjugate diameters:

(P1) Each ellipse inscribed in P\ has tangency chords which are parallel to
one of the diagonals of P.

(P2) Each ellipse inscribed in P\ has a pair of conjugate diameters which
are parallel to the diagonals of P.

Note that when we say that two lines are parallel, we include the possibilty
that they are equal, which does, in fact, occur for (P2).

This author is not sure if P1 is known at all, while P2 appears to be known(%
\cite{PA}) only if $E_{0}$ is the ellipse of maximal area inscribed in P.
One of the purposes of this paper is to examine P1 and P2 for a larger class
of convex quadrilaterals, which we call midpoint diagonal
quadrilaterals(defined below).

\textbf{Definition:} A convex quadrilateral, $Q$, is called a midpoint
diagonal quadrilateral if the intersection point of the diagonals of $Q$
coincides with the midpoint of at least one of the diagonals of $Q$.

We show that not only do P1 and P2 each hold for midpoint diagonal
quadrilaterals, but that if $Q$ is \textbf{not }a midpoint diagonal
quadrilateral, then no ellipse inscribed in $Q$ satisfies P1 or P2. Hence
each of these properties completely characterizes the class of midpoint
diagonal quadrilaterals(see Theorems \ref{T2} and \ref{T1} below), and thus
they are a generalization of parallelograms in this sense. A parallelogram,
P, is a special case of a midpoint diagonal quadrilateral since the
diagonals of P\ bisect one another. Equivalently, if $Q$ is not a
parallelogram, then $Q$ is a midpoint diagonal quadrilateral if and only if
the line, $L_{Q}$, thru the midpoints of the diagonals of $Q$ contains one
of the diagonals of $Q$. The line $L_{Q}$\ plays an important role for
ellipses inscribed in quadrilaterals due to the following well--known
result(see \cite{C} for a proof).

\begin{theorem}
(Newton)\label{Newton}: Let $M_{1}$ and $M_{2}$ be the midpoints of the
diagonals of a quadrilateral, $Q$. If $E_{0}$ is an ellipse inscribed in $Q$%
, then the center of $E_{0}$ must lie on the open line segment, $\overline{%
M_{1}M_{2}}$, connecting $M_{1}$ and $M_{2}$.
\end{theorem}

\begin{remark}
If $Q$ is a parallelogram, then the diagonals of $Q$ intersect at the
midpoints of the diagonals of $Q$, and thus $\overline{M_{1}M_{2}}$ is
really just one point.
\end{remark}

\begin{remark}
By Theorem \ref{Newton}, $Q$ is a midpoint diagonal quadrilateral if and
only if the center of any ellipse inscribed in $Q\ $lies on one of the
diagonals of $Q$.
\end{remark}

If $E_{0}$ is an ellipse which is not a circle, then $E_{0}$ has a unique
set of conjugate diameters, $l$ and $m$, where $\left\vert l\right\vert
=\left\vert m\right\vert $. These are called equal conjugate diameters of $%
E_{0}$.\ By Theorem \ref{T1}(i), each ellipse inscribed in a midpoint
diagonal quadrilateral, $Q$, has conjugate diameters parallel to the
diagonals of $Q$. In particular, Theorem \ref{T1}(i) applies to the unique
ellipse of minimal eccentricity, $E_{I}$, inscribed in $Q$. However, we
prove(Theorem \ref{T3}) a stronger result: The \textbf{equal} conjugate
diameters of $E_{I}$\ are parallel to the diagonals of $Q$.

\textbf{Useful Results on Ellipses and Quadrilaterals}

We now state a result, without proof, about when a quadratic equation in $x$
and $y$ yields an ellipse. The first condition ensures that the conic is an
ellipse, while the second condition ensures that the conic is nondegenerate.

\begin{lemma}
\label{L1}The equation $Ax^{2}+Bxy+Cy^{2}+Dx+Ey+F=0$, with $A,C>0$, is the
equation of an ellipse if and only if $\Delta >0$ and $\delta >0$, where

\begin{eqnarray}
\Delta &=&4AC-B^{2}\ \text{and}  \label{delta} \\
\delta &=&CD^{2}+AE^{2}-BDE-F\Delta \text{.}  \notag
\end{eqnarray}
\end{lemma}

The following lemma allows us to express the eccentricity and center of an
ellipse as a function of the coefficients of an equation of that ellipse.

\begin{lemma}
\label{L4}Suppose that $E_{0}$ is an ellipse with equation $%
Ax^{2}+Bxy+Cy^{2}+Dx+Ey+F=0$. Let $a$ and $b$ denote the lengths of the
semi--major and semi--minor axes, respectively, of $E_{0}$. Let $%
(x_{0},y_{0})$ denote the center of $E_{0}$ and let $\Delta $ and $\delta $
be as in (\ref{delta}). Then%
\begin{equation}
\dfrac{b^{2}}{a^{2}}=\dfrac{A+C-\sqrt{(A-C)^{2}+B^{2}}}{A+C+\sqrt{%
(A-C)^{2}+B^{2}}}\text{\ and}  \label{eccab}
\end{equation}

\begin{equation}
x_{0}=\dfrac{BE-2CD}{\Delta },y_{0}=\dfrac{BD-2AE}{\Delta }\text{.}
\label{center}
\end{equation}
\end{lemma}

\begin{proof}
Let $\mu =\dfrac{4\delta }{\Delta ^{2}}$. By \cite{S}, 
\begin{eqnarray}
a^{2} &=&\dfrac{1}{2}\mu {\large (}A+C+\sqrt{(A-C)^{2}+B^{2}}{\large )},
\label{absq} \\
b^{2} &=&\dfrac{1}{2}\mu {\large (}A+C-\sqrt{(A-C)^{2}+B^{2}}{\large )}\text{%
.}  \notag
\end{eqnarray}%
Note that $\mu >0$ by Lemma \ref{L1}. (\ref{eccab}) then follows immediately
from (\ref{absq}). We omit the details for (\ref{center}).
\end{proof}

Throughout the paper we let $L_{Q}$ denote the \textit{line} thru the
midpoints of a given quadrilateral, $Q$, and we define an affine
transformation, $T:R^{2}\rightarrow R^{2}$ to be the map $T(\hat{x})=A\hat{x}%
+\hat{b}$, where $A$ is an invertible $2\times 2$ matrix. Note that affine
transformations map lines to lines, parallel lines to parallel lines, and
preserve ratios of lengths along a given line. Also, the family of ellipses,
tangent lines to ellipses, and four--sided convex polygons are preserved
under affine transformations.

\qquad A quadrilateral which has an incircle, i.e., one for which a single
circle can be constructed which is tangent to all four sides, is called a
tangential quadrilateral. A quadrilateral which has perpendicular diagonals
is called an orthodiagonal quadrilateral. The following lemmas show that
affine transformations preserve the class of midpoint diagonal
quadrilaterals and send conjugate diameters to conjugate diameters. The
proofs follow immediately from the properties of affine transformations.

\begin{lemma}
\label{L13}Let $T:R^{2}\rightarrow R^{2}$ be an affine transformation and
let $Q$ be a midpoint diagonal quadrilateral. Then $Q^{\prime }=T(Q)$ is
also a midpoint diagonal quadrilateral.
\end{lemma}

\begin{lemma}
\label{L2}Let $T:R^{2}\rightarrow R^{2}$ be an affine transformation and
suppose that $l$ and $m$ are conjugate diameters of an ellipse, $E_{0}$.
Then $l^{\prime }=T(l)$ and $m^{\prime }=T(m)$ are conjugate diameters of $%
E_{0}^{\prime }=T(E_{0})$.
\end{lemma}

The following lemma shows that the scaling transformations preserve the
eccentricity of ellipses, as well as the property of the equal conjugate
diameters of an ellipse being parallel to the diagonals of $Q$.

\begin{lemma}
\label{eccinv}Let $T$ be the nonsingular affine transformation given by $%
T(x,y)=(kx,ky),k\neq 0$.

(i) Then $E_{0}$ and $T(E_{0})$ have the same eccentricity for any ellipse, $%
E_{0}$.

(ii) If $E_{0}$ is an ellipse which is not a circle and if the equal
conjugate diameters of $E_{0}$\ are parallel to the diagonals of $Q$, then
the equal conjugate diameters of $T(E_{0})$\ are parallel to the diagonals
of $T(Q)$.
\end{lemma}

\begin{proof}
(i) follows immediately and we omit the proof. To prove (ii), suppose that $%
l $ and $m$ are equal conjugate diameters of an ellipse, $E_{0}$, which are
parallel to the diagonals of $Q$. By Lemma \ref{L2}, $l^{\prime }=T(l)$ and $%
m^{\prime }=T(m)$ are conjugate diameters of $E_{0}^{\prime }=T(E_{0})$.
Since $\left\vert T(P_{1})T(P_{2})\right\vert =k\left\vert
P_{1}P_{2}\right\vert $ for any two points $P_{1}=(x_{1},y_{1})$ and $%
P_{2}=(x_{2},y_{2})$, $l^{\prime }\ $and $m^{\prime }\ $are equal conjugate
diameters of $T(E_{0})$. Since affine transformations take parallel lines to
parallel lines, $l^{\prime }\ $and $m^{\prime }\ $are parallel to the
diagonals of $T(Q)$. That proves (ii).
\end{proof}

\qquad The following lemma shows when a trapezoid can be a midpoint diagonal
quadrilateral.

\begin{lemma}
\label{mdqtrap}Suppose that $Q$\ is a midpoint diagonal quadrilateral which
is also a trapezoid. Then $Q$ is a parallelogram.
\end{lemma}

\begin{proof}
We use proof by contradiction. So suppose that $Q$\ is a midpoint diagonal
quadrilateral which is a trapezoid, but which is \textit{not} a
parallelogram. By affine invariance, we may assume that $Q$ is the trapezoid
with vertices $(0,0),(1,0),(0,1)$, and $(1,t),0<t\neq 1$. The diagonals of $%
Q $ are then the open line segments $D_{1}$: $y=tx$ and $D_{2}$: $y=1-x$,
each with $0<x<1$. The midpoints of the diagonals are $M_{1}=\left( \dfrac{1%
}{2},\dfrac{1}{2}\right) $ and $M_{2}=\left( \dfrac{1}{2},\dfrac{t}{2}%
\right) $, and the diagonals intersect at $P=\left( \dfrac{1}{1+t},\dfrac{t}{%
1+t}\right) $. Now $M_{1}=P\iff t=1$ and $M_{2}=P\iff t=1$, which
contradicts the assumption that $t\neq 1$. Hence $Q$\ is not a midpoint
diagonal quadrilateral.
\end{proof}

\begin{remark}
We use the notation $Q(A_{1},A_{2},A_{3},A_{4})$ to denote the quadrilateral
with vertices $A_{1},A_{2},A_{3}$, and $A_{4}$, starting with $A_{1}=$ lower
left corner and going clockwise. Denote the sides of $%
Q(A_{1},A_{2},A_{3},A_{4})$ by $S_{1},S_{2},S_{3}$, and $S_{4}$, going
clockwise and starting with the leftmost side, $S_{1}$. Denote the lengths
of the sides of $Q(A_{1},A_{2},A_{3},A_{4})$ by $a=\left\vert
A_{1}A_{4}\right\vert ,b=\left\vert A_{1}A_{2}\right\vert ,c=\left\vert
A_{2}A_{3}\right\vert $, and $d=\left\vert A_{3}A_{4}\right\vert $. Finally,
denote the diagonals of $Q(A_{1},A_{2},A_{3},A_{4})$ by $D_{1}=\overline{%
A_{1}A_{3}}$ and $D_{2}=\overline{A_{2}A_{4}}$.
\end{remark}

\medskip We note here that there are two types of midpoint diagonal
quadrilaterals: Type 1, where the diagonals intersect at the midpoint of $%
D_{2}$ and Type 2, where the diagonals intersect at the midpoint of $D_{1}$.

\textbf{Notation:} The lines containing the diagonal line segments, $D_{1}$
and $D_{2}$, of any quadrilateral are denoted by $\overleftrightarrow{D_{1}}$
and $\overleftrightarrow{D_{2}}$.

Given a convex quadrilateral, $Q=Q(A_{1},A_{2},A_{3},A_{4})$, which is 
\textbf{not} a parallelogram, it will simplify our work below to consider
quadrilaterals with a special set of vertices. In particular, there is an
affine transformation which sends $A_{1},A_{2}$, and $A_{4}$ to the points $%
(0,0),(0,1)$, and $(1,0)$, respectively. It then follows that $A_{3}=(s,t)$
for some $s,t>0$. We thus let $Q_{s,t}$ denote the quadrilateral with
vertices $(0,0),(0,1),(s,t)$, and $(1,0)$. Since $Q_{s,t}$ is convex, $s+t>1$%
. Also, if $Q$ has a pair of parallel vertical sides, first rotate
counterclockwise by $90^{\circ }$, yielding a quadrilateral with parallel
horizontal sides. Since we are assuming that $Q$ is not a parallelogram, we
may then also assume that $Q_{s,t}$ does not have parallel vertical sides
and thus $s\neq 1$. Summarizing, we have

\begin{proposition}
\label{P4}Suppose that $Q$ is a convex quadrilateral which is \textbf{not} a
parallelogram. Then there is an affine transformation which sends $Q$ to the
quadrilateral 
\begin{gather}
Q_{s,t}=Q(A_{1},A_{2},A_{3},A_{4}),  \notag \\
A_{1}=(0,0),A_{2}=(0,1),A_{3}=(s,t),A_{4}=(1,0)\text{,}  \notag \\
\text{with }(s,t)\in G\text{, where}  \label{Qst}
\end{gather}%
\begin{equation}
G=\left\{ (s,t):s,t>0,s+t>1,s\neq 1\right\} \text{.}  \label{G}
\end{equation}
\end{proposition}

Since the midpoints of the diagonals of $Q_{s,t}$\ are $M_{1}=\left( \dfrac{1%
}{2},\dfrac{1}{2}\right) $ and $M_{2}=\left( \dfrac{s}{2},\dfrac{t}{2}%
\right) $, by Theorem \ref{Newton}, the center of any ellipse, $E_{0}$,
inscribed in $Q_{s,t}$ must lie on the open line segment $\left\{ {\large (}%
h,L_{Q}(h){\large )}:h\in I\right\} $, where 
\begin{eqnarray}
L_{Q}(x) &=&\dfrac{1}{2}\dfrac{s-t+2x(t-1)}{s-1},  \label{LQ} \\
I &=&\left\{ 
\begin{array}{ll}
\left( \dfrac{s}{2},\dfrac{1}{2}\right) & \text{if }s<1 \\ 
\left( \dfrac{1}{2},\dfrac{s}{2}\right) & \text{if }s\geq 1\text{.}%
\end{array}%
\right. \text{.}  \notag
\end{eqnarray}%
We now answer the following important question: How does one find the
equation of an ellipse, $E_{0}$, inscribed in $Q_{s,t}$ and the points of
tangency of $E_{0}$ with $Q_{s,t}$ ? We sketch the derivation of the
equation and points of tangency now. First, since $E_{0}$ has center $%
{\large (}h,L_{Q}(h){\large )},h\in I$, one may write the equation of $E_{0}$
in the form 
\begin{equation}
(x-h)^{2}+B(x-h){\large (}y-L_{Q}(h){\large )}+C{\large (}y-L_{Q}(h){\large )%
}^{2}+F=0\text{.}  \label{elleq1}
\end{equation}%
Throughout the rest of the paper we denote the open unit interval by 
\begin{equation*}
J=(0,1)\text{.}
\end{equation*}%
Now suppose that $E_{0}$ is tangent to $Q_{s,t}$ at the points $P_{q}=(q,0)$
and $P_{v}=(0,v)$, where $q,v\in J$. Differentiating (\ref{elleq1}) with
respect to $x$ and plugging in $P_{q}$ and $P_{v}$ yields 
\begin{gather}
q-h=\dfrac{BL_{Q}(h)}{2},  \label{uv} \\
v-L_{Q}(h)=\dfrac{Bh}{2C}\text{.}  \notag
\end{gather}%
Plugging in $P_{q}$ and $P_{v}$ into (\ref{elleq1}) yields $%
(q-h)^{2}-BL_{Q}(h)(q-h)+C{\large (}L_{Q}(h){\large )}^{2}+F=0$ and $h^{2}-Bh%
{\large (}v-L_{Q}(h){\large )}+C{\large (}v-L_{Q}(h){\large )}^{2}+F=0$. By (%
\ref{uv}), we have $F=\dfrac{h^{2}}{4C}\left( B^{2}-4C\right) $ and $F=%
\dfrac{L_{Q}^{2}(h)}{4}\left( B^{2}-4C\right) $. Using both expressions for $%
F$ gives

\begin{equation}
C=\dfrac{h^{2}}{L_{Q}^{2}(h)}\text{.}  \label{C}
\end{equation}%
Now by (\ref{uv}) again 
\begin{equation}
B=\dfrac{2(q-h)}{L_{Q}(h)}\text{.}  \label{B}
\end{equation}%
(\ref{uv}), (\ref{B}), and (\ref{C}) then imply that 
\begin{equation}
v=\dfrac{qL_{Q}(h)}{h}\text{.}  \label{uv2}
\end{equation}%
Substituting (\ref{B}) and (\ref{C}) into $F=\dfrac{h^{2}}{4C}\left(
B^{2}-4C\right) $ yields $F=\allowbreak \allowbreak q^{2}-2qh$. (\ref{elleq1}%
) then becomes

\begin{equation}
(x-h)^{2}+\dfrac{2(q-h)}{L_{Q}(h)}(x-h){\large (}y-L_{Q}(h){\large )}+\dfrac{%
h^{2}}{L_{Q}^{2}(h)}{\large (}y-L_{Q}(h){\large )}^{2}+q^{2}-2qh=0\text{.}
\label{elleq2}
\end{equation}

\begin{remark}
Using Lemma \ref{L1}, it is not hard to show that (\ref{elleq2}) defines the
equation of an ellipse for any $h\in I$.
\end{remark}

Finally, we want to find $h$ in terms of $q$, which makes the final equation
simpler than expressing everything in terms of $h$. One way to do this is to
use the following well--known theorem of Marden(see \cite{M}).

\begin{theorem}
(\textbf{Marden):} Let $F(z)=\dfrac{t_{1}}{z-z_{1}}+\dfrac{t_{2}}{z-z_{2}}+%
\dfrac{t_{3}}{z-z_{3}}$, $\tsum\limits_{k=1}^{3}t_{k}=1$, and let $Z_{1}$
and $Z_{2}$ denote the zeros of $F(z)$. Let $L_{1},L_{2},L_{3}$ be the line
segments connecting $z_{2}$ \& $z_{3}$, $z_{1}$ \& $z_{3}$, and $z_{1}$\& $%
z_{2}$, respectively. If $t_{1}t_{2}t_{3}>0$, then $Z_{1}$ and $Z_{2}$ are
the foci of an ellipse which is tangent to $L_{1},L_{2}$, and $L_{3}$ at the
points $\zeta _{1}=\dfrac{t_{2}z_{3}+t_{3}z_{2}}{t_{2}+t_{3}}$, $\zeta _{2}=%
\dfrac{t_{1}z_{3}+t_{3}z_{1}}{t_{1}+t_{3}}$, and $\zeta _{3}=\dfrac{%
t_{1}z_{2}+t_{2}z_{1}}{t_{1}+t_{2}}$, respectively.
\end{theorem}

Applying Marden's Theorem to the triangle $\Delta A_{2}A_{3}A_{5}$, where $%
A_{5}=\left( 0,-\dfrac{t}{s-1}\right) $, one can show that $E_{0}$ is
tangent to $Q_{s,t}$ at the point $\left( \dfrac{s-2h}{2(t-1)h+s-t},0\right) 
$. Many of the details of this can be found in \cite{H2}. Hence $q=\dfrac{%
s-2h}{2(t-1)h+s-t}$, which implies that 
\begin{equation}
h=\dfrac{1}{2}\dfrac{q(t-s)+s}{q(t-1)+1}\text{.}  \label{hz}
\end{equation}%
Substituting for $h$ in (\ref{elleq2}) using (\ref{hz}) and (\ref{LQ}), (\ref%
{elleq2}) becomes

\begin{gather}
t^{2}x^{2}+{\large (}4q^{2}(t-1)t+2qt(s-t+2)-2st{\large )}xy+  \label{elleq3}
\\
{\large (}(1-q)s+qt{\large )}^{2}y^{2}-2qt^{2}x-2qt{\large (}(1-q)s+qt%
{\large )}y+q^{2}t^{2}=0\text{.}  \notag
\end{gather}

One point of tangency is of course given by $(q,0)$. Using (\ref{elleq3}),
it is then not difficult to find the other points of tangency, which is
given in the following proposition(we have relabeled $P_{q}$ and $P_{v}$).

\begin{proposition}
\label{P1}Suppose that $E_{0}$ is an ellipse inscribed in $Q_{s,t}$. Then $%
E_{0}$ is tangent to the four sides of $Q_{s,t}$\ at the points $%
q_{1}=\left( 0,\dfrac{qt}{(t-s)q+s}\right) \in S_{1},q_{2}=\left( \dfrac{%
(1-q)s^{2}}{(t-1)(s+t)q+s},\dfrac{t{\large (}s+q(t-1){\large )}}{%
(t-1)(s+t)q+s}\right) \in S_{2},$

$q_{3}=\left( \dfrac{s+q(t-1)}{(s+t-2)q+1},\dfrac{(1-q)t}{(s+t-2)q+1}\right)
\in S_{3}$, and $q_{4}=(q,0)\in S_{4}$, $q\in J$.
\end{proposition}

\begin{remark}
It is not hard to show that each of the denominators of the $q_{j}$ above
are non--zero.
\end{remark}

Finally we state the analogy of Proposition \ref{P1} for parallelograms. A
slightly different version was proven in \cite{H4}. We omit the details of
the proof.

\begin{proposition}
\label{P3}Let P\ be the parallelogram with vertices $%
A_{1}=(-l-d,-k),A_{2}=(-l+d,k),A_{3}=(l+d,k)$, and $A_{4}=(l-d,-k)$, where $%
l,k>0,d<l$. If $E_{0}$ is an ellipse inscribed in P, then $E_{0}$ is tangent
to the four sides of P\ at the points $q_{1}=\left( -l+dv,kv\right) \in S_{1}
$, $q_{2}=(-lv+d,k)\in S_{2}$, $q_{3}=\left( l-dv,-kv\right) \in S_{3}$, and 
$q_{4}=(lv-d,-k)\in S_{4}$.
\end{proposition}

\textbf{Tangency Chords Parallel to the Diagonals}

\qquad The following lemma gives necessary and sufficient conditions for the
quadrilateral $Q_{s,t}$ given in (\ref{Qst}) to be a midpoint diagonal
quadrilateral.

\begin{lemma}
\label{L6}(i) $Q_{s,t}$\ is a type 1 midpoint diagonal quadrilateral if and
only if $s=t$.

(ii) $Q_{s,t}$\ is a type 2 midpoint diagonal quadrilateral if and only if\ $%
s+t=2$.
\end{lemma}

\begin{proof}
The diagonal lines of $Q_{s,t}$ are $\overleftrightarrow{D_{1}}$: $y=\dfrac{t%
}{s}x$ and $\overleftrightarrow{D_{2}}$: $y=1-x$, and they intersect at the
point $P=\left( \dfrac{s}{s+t},\dfrac{t}{s+t}\right) $. The midpoints of the
diagonal line segments $D_{1}$ and $D_{2}$\ are $M_{1}=\left( \dfrac{s}{2},%
\dfrac{t}{2}\right) $ and $M_{2}=\left( \dfrac{1}{2},\dfrac{1}{2}\right) $,
respectively. Now $M_{2}=P\iff \dfrac{s}{s+t}=\dfrac{1}{2}$ and $\dfrac{t}{%
s+t}=\dfrac{1}{2}$, both of which hold if and only if $s=t$. That proves
(i). $M_{1}=P\iff \dfrac{s}{s+t}=\dfrac{1}{2}s$ and $\dfrac{t}{s+t}=\dfrac{1%
}{2}t$, both of which hold if and only if $s+t=2$. That proves (ii).
\end{proof}

Now recall property P1 from the introduction: (P1) Each ellipse inscribed in
P\ has tangency chords which are parallel to one of the diagonals of P. The
following theorem shows that P1 completely characterizes the class of
midpoint diagonal quadrilaterals.

\begin{theorem}
\label{T2}Suppose that $E_{0}$ is an ellipse inscribed in a convex
quadrilateral $Q=Q(A_{1},A_{2},A_{3},A_{4})$. Let $q_{j}\in S_{j},j=1,2,3,4$
denote the points of tangency of $E_{0}$ with $Q$, and let $D_{1}=\overline{%
A_{1}A_{3}}$ and $D_{2}=\overline{A_{2}A_{4}}$ denote the diagonals of $Q$.

(i) If $Q$ is a type 1 midpoint diagonal quadrilateral, then $%
\overleftrightarrow{q_{2}q_{3}}$ and $\overleftrightarrow{q_{1}q_{4}}$ are
parallel to $D_{2}$.

(ii) If $Q$ is a type 2 midpoint diagonal quadrilateral, then $%
\overleftrightarrow{q_{1}q_{2}}$ and $\overleftrightarrow{q_{3}q_{4}}$ are
parallel to $D_{1}$.

(iii) If $Q$ is \textbf{not }a midpoint diagonal quadrilateral, then neither 
$\overleftrightarrow{q_{1}q_{2}}$ nor $\overleftrightarrow{q_{3}q_{4}}$ are
parallel to $D_{1}$, and neither $\overleftrightarrow{q_{2}q_{3}}$ nor $%
\overleftrightarrow{q_{1}q_{4}}$ are parallel to $D_{2}$.
\end{theorem}

\begin{proof}
\textbf{Case 1:} $Q$ is not a parallelogram.

Then by Proposition \ref{P4}, we may assume that $Q=$ $Q_{s,t}$\ with
diagonal lines $\overleftrightarrow{D_{1}}$: $y=\dfrac{t}{s}x$ and $%
\overleftrightarrow{D_{2}}$: $y=1-x$. Using Proposition \ref{P1}, after some
simplification we have:

slope of $\overleftrightarrow{q_{1}q_{2}}=\dfrac{t{\large (}2q(t-1)+s{\large %
)}}{s{\large (}(t-s)q+s{\large )}}$, so that the slope of $%
\overleftrightarrow{q_{1}q_{2}}=\dfrac{t}{s}\iff \allowbreak \dfrac{2(t-1)q+s%
}{(t-s)q+s}=1\iff \allowbreak \left( s+t-2\right) q=0\iff s+t=2$ since $%
q=0\notin J$.

slope of $\overleftrightarrow{q_{3}q_{4}}=\dfrac{t}{(s+t-2)q+s}$, so that
the slope of $\overleftrightarrow{q_{3}q_{4}}=\dfrac{t}{s}\iff $

$\dfrac{1}{(s+t-2)q+s}=\dfrac{1}{s}\iff (s+t-2)q=0\iff s+t=2$ since $%
q=0\notin J$.

slope of $\overleftrightarrow{q_{2}q_{3}}=-\dfrac{t{\large (}2(t-1)q+s-t+1%
{\large )}}{\left( s^{2}+t^{2}-s-t\right) q-s^{2}+st+s}$, so that the slope
of $\overleftrightarrow{q_{2}q_{3}}=-1\iff $

$t{\large (}2(t-1)q+s-t+1{\large )}=\left( s^{2}+t^{2}-s-t\right)
q-s^{2}+st+s\iff \allowbreak $

$\left( s+t-1\right) \left( s-t\right) \left( q-1\right) =0\iff s=t$ since $%
q=1\notin J$ and $s+t\neq 1$.

slope of $\overleftrightarrow{q_{1}q_{4}}=\dfrac{t}{(s-t)q-s}$, so that the
slope of $\overleftrightarrow{q_{1}q_{4}}=-1\iff $

$(s-t)q-s=-t\iff \allowbreak \left( q-1\right) \left( s-t\right) =0\iff s=t$
since $q=1\notin J$.

Theorem \ref{T2} then follows from Lemma \ref{L6}.

\textbf{Case 2:} $Q$ is a parallelogram.

As noted in the introduction, Theorem \ref{T2} is probably known in this
case, and there are undoubtedly other ways to prove it for parallelograms.
Using Proposition \ref{P3}, it follows easily that the slope of $%
\overleftrightarrow{q_{1}q_{2}}=$ the slope of $\overleftrightarrow{%
q_{3}q_{4}}=$ $\dfrac{k}{l+d}$ and the slope of $\overleftrightarrow{%
q_{2}q_{3}}=$ the slope of $\overleftrightarrow{q_{1}q_{4}}=\dfrac{k}{d-l}$.
Since the diagonal lines $Q$ are $\overleftrightarrow{D_{1}}D_{1}$: $y=k+%
\dfrac{k}{l+d}(x-l-d)$ and $\overleftrightarrow{D_{2}}$: $y=k+\dfrac{k}{d-l}%
(x+l-d)$, and a parallelogram is a special case of a midpoint diagonal
quadrilateral, that proves Theorem \ref{T2} for case 2. Note that one could
map $Q$ to the unit square and then use a simplified version of Proposition %
\ref{P3}, but that does not simplify the proof very much.
\end{proof}

\qquad Now recall that the lengths of the sides of $%
Q(A_{1},A_{2},A_{3},A_{4})$ are denoted by $a=\left\vert
A_{1}A_{4}\right\vert ,b=\left\vert A_{1}A_{2}\right\vert ,c=\left\vert
A_{2}A_{3}\right\vert $, and $d=\left\vert A_{3}A_{4}\right\vert $.

\begin{lemma}
\label{L3}Suppose that $Q=Q(A_{1},A_{2},A_{3},A_{4})$ is both a tangential
and a midpoint diagonal quadrilateral. Then $Q$ is an orthodiagonal
quadrilateral.
\end{lemma}

\begin{remark}
We actually prove more--that $Q$ is a \textbf{kite}. That is, that two pairs
of adjacent sides of $Q$ are equal.
\end{remark}

\begin{proof}
Since $Q$ is tangential, there is a circle, $E_{0}$, inscribed in $Q$. Let $%
q_{j}\in S_{j},j=1,2,3,4$ denote the points of tangency of $E_{0}$ with $Q$.
Define the triangles $T_{1}=\Delta q_{4}A_{1}q_{1}$ and $T_{2}=\Delta
A_{4}A_{1}A_{2}$, and define the lines $L_{1}=\overleftrightarrow{q_{1}q_{4}}
$ and $L_{2}=\overleftrightarrow{q_{2}q_{3}}$. Suppose first that $Q$ is a
type 1 midpoint diagonal quadrilateral. Then $L_{1}\parallel D_{2}$ by
Theorem \ref{T2}(i), which implies that $T_{1}$ and $T_{2}$ are similar
triangles. Also, since $E_{0}$ is a circle, $\left\vert
A_{1}q_{4}\right\vert =\left\vert A_{1}q_{1}\right\vert $, which implies
that $T_{1}$ is isoceles. Hence $T_{2}$ is also isoceles with $b=a$. In a
similar fashion, one can show that $c=d$ using the fact that $L_{2}\parallel 
$ $D_{2}$. Thus $a^{2}+c^{2}=b^{2}+d^{2}$, which implies that $Q$ is an
orthodiagonal quadrilateral. The proof when $Q$ is a type 2 midpoint
diagonal quadrilateral is similar and we omit the details.
\end{proof}

\qquad We now prove a result somewhat similar to Lemma \ref{L3}.

\begin{lemma}
\label{L5}Suppose that $Q=Q(A_{1},A_{2},A_{3},A_{4})$ is both a tangential
and an orthodiagonal quadrilateral. Then $Q$ is a midpoint diagonal
quadrilateral.
\end{lemma}

\begin{proof}
Since $Q$ is\ tangential, there is a circle, $E_{0}$, inscribed in $Q$ and $%
a+c=b+d$, which implies that $d=a+c-b$. Since $Q$ is orthodiagonal, $%
a^{2}+c^{2}=b^{2}+d^{2}$. Hence $b^{2}+(a+c-b)^{2}-a^{2}-c^{2}=0$, which
implies that $\allowbreak 2\left( b-c\right) \left( b-a\right) =0$, and so $%
a=b$ and/or $b=c$. We prove the case when $a=b$. Let $q_{j}\in
S_{j},j=1,2,3,4$ denote the points of tangency of $E_{0}$ with $Q$. Then the
triangle $T_{1}=\Delta q_{4}A_{1}q_{1}$ is isoceles since $\left\vert
A_{1}q_{4}\right\vert =\left\vert A_{1}q_{1}\right\vert $, and the triangle $%
T_{2}=\Delta A_{4}A_{1}A_{2}$ is isoceles since $a=b$. Thus $T_{1}$ and $%
T_{2}$ are similar triangles, which implies that the line $%
\overleftrightarrow{q_{1}q_{4}}$ is parallel to $D_{2}$. By Theorem \ref{T2}%
(iii), $Q$ is a midpoint diagonal quadrilateral.
\end{proof}

\textbf{Conjugate Diameters Parallel to the Diagonals}

Recall property P2 from the introduction: (P2) Each ellipse inscribed in P\
has a pair of conjugate diameters which are parallel to the diagonals of P.
The following theorem shows that P2 completely characterizes the class of
midpoint diagonal quadrilaterals.

\begin{theorem}
\label{T1}(i) Suppose that $Q$ is a midpoint diagonal quadrilateral. Then 
\textbf{each} ellipse inscribed in $Q$ has a unique pair of conjugate
diameters parallel to the diagonals of $Q$.

(ii) Suppose that $Q$ is \textbf{not} a midpoint diagonal quadrilateral.
Then \textbf{no} ellipse inscribed in $Q$ has conjugate diameters parallel
to the diagonals of $Q$.
\end{theorem}

\begin{proof}
Let $E_{0}$ be an ellipse inscribed in $Q$ and let $D_{1}$ and $D_{2}$
denote the diagonals of $Q$. Use an affine transformation, $T$, to map $%
E_{0} $ to a circle, $E_{0}^{\prime }$, inscribed in the tangential
quadrilateral, $Q^{\prime }=T(Q)$. Let $L_{1}$ be a diameter of $E_{0}\ $%
parallel to $D_{1}$. $T$ maps $L_{1}$ to a diameter, $L_{1}^{\prime }$, of $%
E_{0}^{\prime }\ $parallel to one of the diagonals of $Q^{\prime }$, which
we call $D_{1}^{\prime }$. Let $D_{2}^{\prime }$ be the other diagonal of $%
Q^{\prime } $. Let $L_{2}^{\prime }$ be the diameter of $E_{0}^{\prime }$
perpendicular to $L_{1}^{\prime }$, which implies that $L_{1}^{\prime }$ and 
$L_{2}^{\prime }$ are conjugate diameters since $E_{0}^{\prime }$ is a
circle. By Lemma \ref{L2}, $T^{-1}$ maps $L_{2}^{\prime }$ to $L_{2}$, a
diameter of $E_{0}$ conjugate to $L_{1}$. To prove (i), suppose that $Q$ is
a midpoint diagonal quadrilateral. By Lemma\textbf{\ }\ref{L13}, $Q^{\prime
} $ is also a midpoint diagonal quadrilateral. By Lemma \ref{L3}, $Q^{\prime
}$ is an orthodiagonal quadrilateral, which implies that $D_{1}^{\prime
}\perp D_{2}^{\prime }$. Since $L_{1}^{\prime }\parallel D_{1}^{\prime
},L_{1}^{\prime }\perp L_{2}^{\prime }$, and $D_{1}^{\prime }\perp
D_{2}^{\prime }$, $L_{2}^{\prime }$ must be parallel to $D_{2}^{\prime }$,
which implies that $L_{2}$ is parallel to $D_{2}$ since $T^{-1}$ is an
affine transformation. That proves (i). To prove (ii), suppose that $Q$ is 
\textbf{not} a midpoint diagonal quadrilateral. Since $Q^{\prime }$ is
tangential, if $Q^{\prime }$ were also an orthodiagonal quadrilateral, then
by Lemma \ref{L5}, $Q^{\prime }$ would be a midpoint diagonal quadrilateral.
Hence $Q^{\prime }$ cannot be an orthodiagonal quadrilateral, which implies
that $D_{1}^{\prime }$ is not perpendicular to $D_{2}^{\prime }$. Now if $%
L_{2}^{\prime }$ were parallel to $D_{2}^{\prime }$, then it would follow
that $D_{1}^{\prime }\perp D_{2}^{\prime }$ since $L_{1}^{\prime }\parallel
D_{1}^{\prime }$ and $L_{1}^{\prime }\perp L_{2}^{\prime }$, a
contradiction. Hence $L_{2}^{\prime }\nparallel D_{2}^{\prime }$, which
implies that $L_{2}\nparallel D_{2}$ since $T$ is an affine transformation.
That proves (ii).
\end{proof}

\textbf{Equal Conjugate Diameters and the Ellipse of Minimal Eccentricity}

By Theorem \ref{T1}(i), each ellipse inscribed in a midpoint diagonal
quadrilateral, $Q$, has conjugate diameters parallel to the diagonals of $Q$%
. In particular, this holds for the unique ellipse of minimal
eccentricity(whose existence we prove below), $E_{I}$, inscribed in $Q$.
However, in Theorem \ref{T3}(ii) below, we prove a stronger result for $%
E_{I} $.

\begin{theorem}
\label{T3}

(i) There is a unique ellipse of minimal eccentricity, $E_{I}$, \textit{%
inscribed} in a midpoint diagonal quadrilateral, $Q$.

(ii) Furthermore, the unique pair of conjugate diameters parallel to the
diagonals of $Q$ are \textbf{equal} conjugate diameters of $E_{I}$.
\end{theorem}

\begin{remark}
\label{Re1}Suppose that $Q$\ is a type 1 midpoint diagonal quadrilateral and
let $CD_{1}$\textbf{\ }and\textbf{\ }$CD_{2}$ be the \textit{equal}
conjugate diameters in Theorem \ref{T3}(ii) parallel to the diagonals, $%
D_{1} $\ and $D_{2}$. Let $\overleftrightarrow{CD_{1}}$ and $%
\overleftrightarrow{CD_{2}}$ denote the lines containing $CD_{1}$\textbf{\ }%
and\textbf{\ }$CD_{2} $, respectively. Since $\overleftrightarrow{CD_{1}}$
is parallel to $\overleftrightarrow{D_{1}}$ and $\overleftrightarrow{D_{1}}%
=L $, $\overleftrightarrow{CD_{1}}$ is parallel to $L$. Since $L$ and $%
\overleftrightarrow{CD_{1}}$ each pass through the center of $E_{I},%
\overleftrightarrow{CD_{1}}=L$. Similarly, for type 2 midpoint diagonal
quadrilaterals, $\overleftrightarrow{CD_{2}}=L$.
\end{remark}

\begin{remark}
Theorem \ref{T3}(ii) cannot hold if $Q$ is \textbf{not} a midpoint diagonal
quadrilateral, since in that case no ellipse inscribed in $Q$ has conjugate
diameters parallel to the diagonals of $Q$ by Theorem \ref{T1}(ii). But
Theorem \ref{T3}(ii) implies the following weaker result: The smallest
nonnegative angle between equal conjugate diameters of $E_{I}$ equals the
smallest nonnegative angle between the diagonals of $Q$ when $Q$ is a
midpoint diagonal quadrilateral. This was proven in \cite{H4} for
parallelograms. We do not know if this property of $E_{I}$ can hold if $Q$
is \textbf{not} a midpoint diagonal quadrilateral.
\end{remark}

Before proving Theorem \ref{T3}, we need several preliminary results. We
omit the details for the proof of Theorem \ref{T3} when $Q$ is a
parallelogram. So suppose that $Q$\ is a midpoint diagonal quadrilateral
which is \textbf{not }a parallelogram. By using an \textbf{isometry} of the
plane, we may assume that $Q$ has vertices $(0,0),(0,u),(s,t)$, and $(v,w)$,
where $s,v,u>0,t>w$. To obtain this isometry, first, if $Q$ has a pair of
parallel vertical sides, first rotate counterclockwise by $90^{\circ }$,
yielding a quadrilateral with parallel horizontal sides. Since we are
assuming that $Q$ is not a parallelogram, we may then also assume that $Q$
does not have parallel vertical sides. One can now use a translation, if
necessary, to map the lower left hand corner vertex of $Q$ to $(0,0)$.
Finally a rotation, if necessary, yields vertices $(0,0),(0,u),(s,t)$, and $%
(v,w)$. Note that such a rotation leaves $Q$ without parallel vertical
sides. In addition, by Lemma \ref{eccinv} with $T(x,y)=\dfrac{1}{u}\left(
x,y\right) $, we may also assume that one of the vertices of $Q$ is $(0,1)$.
So we now work with the quadrilateral 
\begin{gather}
Q_{s,t,v,w}=Q(A_{1},A_{2},A_{3},A_{4}),A_{1}=(0,0),  \label{Q} \\
A_{2}=(0,1),A_{3}=(s,t),A_{4}=(v,w)\text{,}  \notag
\end{gather}%
where 
\begin{equation}
s,v>0,t>w,s\neq v\text{.}  \label{R0}
\end{equation}

The sides of $Q_{s,t,v,w}$, going clockwise, are given by $S_{1}=\overline{%
(0,0)\ (0,u)},S_{2}=\overline{(0,1)\ (s,t)},S_{3}=\overline{(s,t)\ (v,w)}$,
and $S_{4}=\overline{(0,0)\ (v,w)}$. By Lemma \ref{L13}, $Q_{s,t,v,w}$ is a
midpoint diagonal quadrilateral, which implies, by Lemma \ref{mdqtrap}, that 
$Q_{s,t,v,w}$\ is \textbf{not }a trapezoid since $Q_{s,t,v,w}$\ is \textbf{%
not }a parallelogram. We find it useful to define the following expressions,
each of which depend on $s,t,v$, and $w$: 
\begin{eqnarray}
f_{1} &=&v(t-1)+(1-w)s,  \notag \\
f_{2} &=&vt-ws,  \label{f123} \\
f_{3} &=&ws-v(t-1)\text{.}  \notag
\end{eqnarray}

\textbullet\ Since $Q_{s,t,v,w}$\ is convex, $(s,t)$ must lie above $%
\overleftrightarrow{(0,1)\ (v,w)}$\ and $(v,w)$ must lie below $%
\overleftrightarrow{(0,0)\ (s,t)}$, which implies that 
\begin{equation}
f_{1}>0\text{ and }f_{2}>0\text{.}  \label{R1}
\end{equation}%
Since no two sides of $Q_{s,t,v,w}$\ are parallel, $S_{2}\nparallel S_{4}$,
which implies that 
\begin{equation}
f_{3}\neq 0\text{.}  \label{R2}
\end{equation}

$M_{1}=\left( \dfrac{1}{2}v,\dfrac{1}{2}(w+1)\right) $ and $M_{2}=\left( 
\dfrac{1}{2}s,\dfrac{1}{2}t\right) $ are the midpoints of the diagonals of $%
Q_{s,t,v,w}$\ and the equation of the line thru $M_{1}$ and $M_{2}$ is

\begin{eqnarray}
y &=&L(x)=\dfrac{t}{2}+\dfrac{w+1-t}{v-s}\left( x-\dfrac{s}{2}\right) ,
\label{L} \\
x &\in &I=\left\{ 
\begin{array}{ll}
(v/2,s/2) & \text{if }v<s \\ 
(s/2,v/2) & \text{if }s<v%
\end{array}%
\right. \text{.}  \notag
\end{eqnarray}%
The diagonal line segments of $Q_{s,t,v,w}$ are $D_{1}=\overline{(0,0)\ (s,t)%
}$ and $D_{2}=\overline{(0,1)\ (v,w)}$. Now let $E_{0}$ be an ellipse
inscribed in $Q_{s,t,v,w}\ $and suppose that $E_{0}$ is tangent to $%
Q_{s,t,v,w}$ at the points $P_{q}=\left( q,\dfrac{w}{v}q\right) \in $ $S_{4}$
and $P_{r}=(0,r)\in S_{1}$, $0<q<v$, $r\in J$. Using these points of
tangency, it is not hard to show that $q=\dfrac{svr}{\allowbreak \left(
s-f_{2}\right) r+f_{2}}$. That leads to Proposition \ref{P2} below, which
gives necessary and sufficient conditions for the general equation of an
ellipse inscribed in $Q_{s,t,v,w}$. It is useful for us to emphasize the
dependence of the coefficients of the general equation on the parameter $r$
in our notation.

\begin{proposition}
\label{P2}Suppose that (\ref{R0}), (\ref{R1}), and (\ref{R2}) hold.

(i) Let $E_{0}$ be an ellipse inscribed in $Q_{s,t,v,w}$. Then for some $%
r\in J$, the general equation of $E_{0}$ is given by $\psi (x,y)=0$, where \ 
\begin{equation}
\psi (x,y)=A(r)x^{2}+B(r)xy+C(r)y^{2}+D(r)x+E(r)y+F(r)\text{,}  \label{psi}
\end{equation}%
and 
\begin{gather}
A(r)={\large (}s^{2}+v^{2}t^{2}+w^{2}s^{2}-2tvs(w+1)+2ws(2v-s){\large )}%
r^{2}+  \notag \\
2v\left( st-2ws-t^{2}v+tsw\right) r+t^{2}v^{2},  \notag \\
B(r)=-2vs{\large (}2r^{2}(v-s)+rs(w+1)+v(t-rt-2r){\large ),}  \label{1} \\
C(r)=v^{2}s^{2},D(r)=2srv{\large (}-rs(w+1)+2ws+tv(r-1){\large ),}  \notag \\
E(r)=-2rv^{2}s^{2},F(r)=r^{2}s^{2}v^{2}\text{.}  \notag
\end{gather}

(ii) Conversely, if for some $r\in J$ the general equation of $E_{0}$ is
given by $\psi (x,y)=0$, where (\ref{psi}) and (\ref{1}) hold, then $E_{0}$
is an ellipse inscribed in $Q_{s,t,v,w}$.
\end{proposition}

\begin{proof}
Using Lemma \ref{L1}, it is not hard to show that $\psi (x,y)=0$ defines the
equation of an ellipse for any $r\in J$. Using standard calculus techniques,
it is also not difficult to show that the ellipse defined by $\psi (x,y)=0$
is inscribed in $Q_{s,t,v,w}$ for any $r\in J$. The converse result, that
any ellipse inscribed in $Q_{s,t,v,w}$ has equation given by $\psi (x,y)=0$
for some $r\in J$ can be proven in a similar fashion to the proof of
Proposition \ref{P1}. We leave the details to the reader.
\end{proof}

The following lemma gives necessary and sufficient conditions for $%
Q_{s,t,v,w}$ to be a midpoint diagonal quadrilateral.

\begin{lemma}
\label{L14}Suppose that (\ref{R0}), (\ref{R1}), and (\ref{R2}) hold.

(i) $Q_{s,t,v,w}$ is a type 1 midpoint diagonal quadrilateral if and only if 
\begin{equation}
vt=(w+1)s\text{.}  \label{v1}
\end{equation}

(ii) $Q_{s,t,v,w}$ is a type 2 midpoint diagonal quadrilateral if and only
if 
\begin{equation}
(t-2)v=(w-1)s\text{.}  \label{v2}
\end{equation}
\end{lemma}

\begin{proof}
$\overleftrightarrow{D_{1}}$ has equation $y=\dfrac{t}{s}x$. Using (\ref{L}%
), $\overleftrightarrow{D_{1}}=L\iff $%
\begin{gather}
\dfrac{w+1-t}{v-s}=\dfrac{t}{s}\ \text{and}  \label{19} \\
\dfrac{t}{2}-\dfrac{s}{2}\dfrac{w+1-t}{v-s}=0\text{.}  \notag
\end{gather}%
It follows easily that (\ref{19}) holds if and only if (\ref{v1}) holds,
which proves (i). The proof of (ii) follows in a similar fashion.
\end{proof}

\textbf{Proof of Theorem \ref{T3}}

\begin{proof}
We assume first that $Q$ is a \textit{tangential} quadrilateral. Then $Q$ is
an orthodiagonal quadrilateral by Lemma\textbf{\ }\ref{L3}, and so the
diagonals of $Q$ are perpendicular. Also, there is a unique circle, $\Phi $,
inscribed in $Q$, which implies that $\Phi $ is the unique ellipse of
minimal eccentricity inscribed in $Q$ since $\Phi $ has eccentricity $0$.
Since any pair of perpendicular diameters of a circle are equal conjugate
diameters, in particular the unique pair which are parallel to the diagonals
of $Q$ are equal conjugate diameters of $\Phi $, and Theorem \ref{T3} holds.
So assume now that $Q$ is \textbf{not }a tangential quadrilateral. It
suffices to assume that $Q=Q_{s,t,v,w}$ and that (\ref{R0}), (\ref{R1}), and
(\ref{R2}) hold. Let $E_{0}$ be an ellipse inscribed in $Q_{s,t,v,w}$. By
Proposition \ref{P4}, the general equation of $E_{0}$ is given by $\psi
(x,y)=0$, where $\psi $ is given by (\ref{psi}) and (\ref{1}) for some $r\in
J$. Let $a$ and $b$ denote the lengths of the semi--major and semi--minor
axes, respectively, of $E_{0}$. Now $\dfrac{b^{2}}{a^{2}}$ is really a
function of $r\in J\ $if we allow $E_{0}$\ to vary over all ellipses
inscribed in $Q_{s,t,v,w}$. By (\ref{eccab}) in Lemma \ref{L4}, $\dfrac{b^{2}%
}{a^{2}}=G(r)$, where $G(r)=\dfrac{A(r)+C(r)-\sqrt{{\large (}A(r)-C(r)%
{\large )}^{2}+{\large (}B(r){\large )}^{2}}}{A(r)+C(r)+\sqrt{{\large (}%
A(r)-C(r){\large )}^{2}+{\large (}B(r){\large )}^{2}}}$. Since the square of
the eccentricity of $E_{0}$ equals $1-\dfrac{b^{2}}{a^{2}}$, it suffices to
maximize $\dfrac{b^{2}}{a^{2}}$. Letting 
\begin{eqnarray}
O(r) &=&A(r)+C(r),  \label{Om} \\
M(r) &=&{\large (}A(r)-C(r){\large )}^{2}+{\large (}B(r){\large )}^{2}\text{,%
}  \notag
\end{eqnarray}%
we have $G(r)=\dfrac{O(r)-\sqrt{M(r)}}{O(r)+\sqrt{M(r)}}$. Since $Q$ is not%
\textbf{\ }a tangential quadrilateral, it follows easily that $Q_{s,t,v,w}$
is also not\textbf{\ }a tangential quadrilateral. If $M(r_{0})=0$ for some $%
r_{0}\in J$, then $A(r_{0})-C(r_{0})=0$ and $B(r_{0})=0$, which implies that
the ellipse inscribed in $Q_{s,t,v,w}$ corresponding to $r_{0}$ is a circle.
But that contradicts the assumption that $Q_{s,t,v,w}$ is not\textbf{\ }a
tangential quadrilateral. Thus $M(r)\neq 0$ for all $r\in J$. Since $M$ is
non--negative, it follows that 
\begin{equation}
M(r)>0,r\in J\text{.}  \label{31}
\end{equation}%
Define the quartic polynomial $\allowbreak $%
\begin{equation*}
N(r)=O^{2}(r)-M(r)\text{.}
\end{equation*}%
After some simplification, $N$ factors as 
\begin{equation}
N(r)=16s^{2}v^{2}r\left( 1-r\right) {\large (}(s-v)r+v{\large )(}(s-v)r+f_{2}%
{\large )}\text{,}  \label{Nfact}
\end{equation}%
and $N$ has roots 
\begin{equation}
r_{1}=0,r_{2}=1,r_{3}=\dfrac{f_{2}}{v-s},r_{4}=\dfrac{v}{v-s}\text{.}
\label{Nroots}
\end{equation}

Note that $r_{3}=r_{4}\iff f_{3}=0$, which cannot hold by (\ref{R2}), and $%
r_{3}\neq 0\neq r_{4}$ since $v\neq 0\neq f_{2}$ by (\ref{R0}) and (\ref{R1}%
). $r_{3}=1\iff f_{1}=0$, which cannot hold by (\ref{R1}), and $r_{4}=1\iff
s=0$, which cannot hold by (\ref{R0}). Thus all roots listed in (\ref{Nroots}%
) are \textit{distinct}. A simple computation yields%
\begin{equation}
G^{\prime }(r)=\dfrac{p(r)}{{\large (}O(r)+\sqrt{M(r)}{\large )}^{2}\sqrt{%
M(r)}}\text{,}  \label{dG}
\end{equation}

where the quartic polynomial $p$ is given by 
\begin{equation}
p(r)=2M(r)O^{\prime }(r)-O(r)M^{\prime }(r)\text{.}  \label{p}
\end{equation}
\end{proof}

\qquad To finish the proof of Theorem \ref{T3}, the following lemmas will be
used to show that $p$ has a unique root in $J$.

\begin{lemma}
\label{N}$N(r)>0\ $on $J$.
\end{lemma}

\begin{proof}
First define the linear function of $r,L(r)=(s-v)r+f_{2}$. $L(0)=f_{2}>0$
and $L(1)=$\ $f_{1}>0$ by (\ref{R1}) and thus $L>0\ $on $J$. Similarly, $%
(s-v)r+v>0$ on $J$. By (\ref{Nfact}), $N>0$ on $J\ $since $r\left(
1-r\right) >0$ on $J$.
\end{proof}

\begin{lemma}
\label{O}$O(r)>0$ on $J$.
\end{lemma}

\begin{proof}
If $O(r_{0})=0$ for some $r_{0}\in J$, then $N(r_{0})=-M(r_{0})\leq 0$ since 
$M(r)\geq 0$ on $J$. That contradicts Lemma \ref{N}. Hence $O(r)$ is nonzero
on $J$. Since $O(0)=\allowbreak \left( s^{2}+t^{2}\right) v^{2}>0$, that
proves Lemma \ref{O}.

Now assume that $Q_{s,t,v,w}$ is a type 1 midpoint diagonal quadrilateral.
In \cite{H2} we proved that there is a unique ellipse of minimal
eccentricity inscribed in any convex quadrilateral, $Q$. The uniqueness for
midpoint diagonal quadrilaterals would then follow from that result.
However, the proof here, specialized for midpoint diagonal quadrilaterals,
is self--contained, uses different methods, and does not require the result
from \cite{H2}. Use (\ref{v1}) to substitute $\dfrac{s\left( w+1\right) }{v}$
for $t$ in Proposition \ref{P2}. Some simplification then yields $%
A(r)=\allowbreak s{\large (}4w\left( v-s\right) r^{2}-4vwr+s\left(
w+1\right) ^{2}{\large )}$, $B(r)=\allowbreak 2sv{\large (}\allowbreak
2\left( s-v\right) r^{2}+2vr-s\left( 1+w\right) {\large )}$, $%
C(r)=\allowbreak s^{2}v^{2}$, $D(r)=\allowbreak 2rs^{2}v\left( w-1\right) $, 
$E(r)=\allowbreak -2rs^{2}v^{2}$, and $F(r)=\allowbreak r^{2}s^{2}v^{2}$.
Using (\ref{Om}) and (\ref{p}), it then follows that $p(r)=-16v^{2}s^{4}%
{\large (}\allowbreak 2\left( s-v\right) r+v{\large )}\alpha (r)$, where 
\begin{equation}
\alpha (r)=2\left( s-v\right) \left( v^{2}+w^{2}+1\right) r^{2}+2v\left(
v^{2}+w^{2}+1\right) r-s{\large (}v^{2}+(w+1)^{2}{\large )}\text{.}
\label{q}
\end{equation}

(\ref{R0}), (\ref{R1}), and (\ref{R2}) now become $s,v>0$(already assumed)
and $\allowbreak 2s-v>0$. Now $\alpha (0)=\allowbreak -s{\large (}%
v^{2}+(w+1)^{2}{\large )}<0$ and $\alpha (1)=s{\large (}v^{2}+(w-1)^{2}%
{\large )}>0$, which implies that $\alpha $ has precisely one root in $J$
since $\alpha $ is a quadratic. By (\ref{dG}), (\ref{31}), and Lemma \ref{O}%
, $G$ is differentiable on $J$. The linear function $2\left( s-v\right) r+v$
is nonzero at $r=0$ since $v>0$ and at $r=1$ since $\allowbreak 2s-v>0$.
Thus the other factors of $p$ are nonzero and hence $p$ has precisely one
root, $r_{1}\in J$, which is also the unique root of $G^{\prime }(r)$ in $J$
by (\ref{dG}). Since $G(r)>0$ on $J$, $G(0)=G(1)=0$, and $G$ is positive in
the interior of $I$ and vanishes at the endpoints of $I$, $G(r_{1})$ must
yield the global maximum of $G$ on $J$. That proves Theorem \ref{T3}(i).
Note that the equation of $E_{I}$ is obtained by letting $r=r_{1}$ in
Proposition \ref{P2}. To prove Theorem \ref{T3}(ii), by Theorem \ref{T1}(i), 
$E_{I}$ has conjugate diameters, $CD_{1}$ and $CD_{2}$, parallel to the
diagonals, $D_{1}$ and $D_{2}$, of $Q_{s,t,v,w}$. $\overleftrightarrow{D_{1}}
$ has equation $y=\dfrac{t}{s}x=\allowbreak \dfrac{w+1}{v}x$ and $%
\overleftrightarrow{D_{2}}$ \ has equation $y=1+\dfrac{w-1}{v}x$. Since $%
Q_{s,t,v,w}$ is type 1, $\overleftrightarrow{D_{1}}=L$, the line thru the
midpoints of $D_{1}$ and $D_{2}$, and thus $L$ has equation $y=\dfrac{w+1}{v}%
x$. It then follows that $\overleftrightarrow{CD_{1}}=L$(see Remark \ref{Re1}%
) and hence $\overleftrightarrow{CD_{1}}$ has equation $y=\dfrac{w+1}{v}x$.
For convenience, we let 
\begin{eqnarray*}
\beta &=&(s-v)r_{1}+v\text{,} \\
\zeta &=&(s-v)r_{1}+s\text{.}
\end{eqnarray*}

By Proposition \ref{P2}, $\dfrac{B(r)E(r)-2C(r)D(r)}{4A(r)C(r)-B^{2}(r)}%
=\allowbreak \dfrac{1}{2}\dfrac{sv}{(s-v)r+v}$. Thus by (\ref{center}) of
Lemma \ref{L4}, $E_{I}$ has center $(x_{0},y_{0})$, where $x_{0}=\dfrac{%
B(r_{1})E(r_{1})-2C(r_{1})D(r_{1})}{4A(r_{1})C(r_{1})-B^{2}(r_{1})}=\dfrac{1%
}{2}\dfrac{sv}{\beta }$. Since $E_{I}$ has center ${\large (}x_{0},L(x_{0})%
{\large )}$ by Newton's Theorem and $L\left( \dfrac{1}{2}\dfrac{sv}{\beta }%
\right) =\dfrac{1}{2}\dfrac{w+1}{v}\dfrac{sv}{\beta }=\dfrac{1}{2}\dfrac{%
s(w+1)}{\beta }$, $E_{I}$ has center $\left( \dfrac{1}{2}\dfrac{sv}{\beta },%
\dfrac{1}{2}\dfrac{s(w+1)}{\beta }\right) $. Since $\overleftrightarrow{%
CD_{2}}$ passes through the center of $E_{I}$ and has the same slope as $%
\overleftrightarrow{D_{2}}$, which is $\dfrac{w-1}{v}$, $\overleftrightarrow{%
CD_{2}}$ has equation $y-\dfrac{1}{2}\dfrac{s(w+1)}{\beta }=\dfrac{w-1}{v}%
\left( x-\dfrac{1}{2}\dfrac{sv}{\beta }\right) $, which simplifies to $y=$ $%
\dfrac{w-1}{v}x+\dfrac{s}{\beta }$.
\end{proof}

\begin{proof}
Now suppose that $CD_{1}$\ intersects $E_{I}$ at the two distinct points $%
P_{1}=(x_{1},y_{1})=\left( x_{1},\dfrac{w+1}{v}x_{1}\right) $ and $%
P_{2}=(x_{2},y_{2})=\left( x_{2},\dfrac{w+1}{v}x_{2}\right) $. Since $P_{1}$
and $P_{2}$ lie on $E_{I}$, by Proposition \ref{P2}, $%
A(r_{1})x_{j}^{2}+B(r_{1})x_{j}y_{j}+C(r_{1})y_{j}^{2}+D(r_{1})x_{j}+E(r_{1})y_{j}+I(r_{1})=0,j=1,2 
$. Substituting $y_{j}=\dfrac{w+1}{v}x_{j}$ yields $%
A(r_{1})x_{j}^{2}+B(r_{1})x_{j}^{2}\left( \dfrac{w+1}{v}\right)
+C(r_{1})x_{j}^{2}\left( \dfrac{w+1}{v}\right)
^{2}+D(r_{1})x_{j}+E(r_{1})\left( \dfrac{w+1}{v}\right)
x_{j}+I(r_{1})=0,j=1,2$, which simplifies to%
\begin{gather}
{\large (}A(r_{1})+\left( \dfrac{w+1}{v}\right) B(r_{1})+\left( \dfrac{w+1}{v%
}\right) ^{2}C(r_{1}){\large )}x_{j}^{2}+  \label{quad1} \\
{\large (}D(r_{1})+\left( \dfrac{w+1}{v}\right) E(r_{1}){\large )}%
x_{j}+I(r_{1})=0,j=1,2\text{.}  \notag
\end{gather}%
If $x_{1}$ and $x_{2}$ are two distinct real roots of any quadratic, then 
\begin{equation*}
x_{2}-x_{1}=\dfrac{\sqrt{\tau }}{a}\Rightarrow (x_{2}-x_{1})^{2}=\dfrac{\tau 
}{a^{2}},
\end{equation*}%
where $\tau =$ discriminant and $a=$ leading coefficient. For the specific
quadratic in the variable $x_{j}$ given in (\ref{quad1}) we have%
\begin{eqnarray*}
\tau &=&{\large (}D(r_{1})+\left( \dfrac{w+1}{v}\right) E(r_{1}){\large )}%
^{2}- \\
&&4F(r_{1}){\large (}A(r_{1})+\left( \dfrac{w+1}{v}\right) B(r_{1})+\left( 
\dfrac{w+1}{v}\right) ^{2}C(r_{1}){\large )}
\end{eqnarray*}%
and 
\begin{equation*}
a=A(r_{1})+\left( \dfrac{w+1}{v}\right) B(r_{1})+\left( \dfrac{w+1}{v}%
\right) ^{2}C(r_{1})\text{.}
\end{equation*}%
Now applying Proposition \ref{P2} and simplfying yields 
\begin{eqnarray*}
\tau &=&16r_{1}^{2}s^{3}v^{2}\left( 1-r_{1}\right) \zeta {\large ,} \\
a &=&4r_{1}s\beta \text{.}
\end{eqnarray*}

Hence $(x_{2}-x_{1})^{2}=\dfrac{\tau }{a^{2}}=\dfrac{v^{2}s(1-r_{1})\zeta }{%
\beta ^{2}}$. Now $y_{2}-y_{1}=\dfrac{w+1}{v}(x_{2}-x_{1})\Rightarrow
(x_{2}-x_{1})^{2}+(y_{2}-y_{1})^{2}=(x_{2}-x_{1})^{2}\left( 1+\left( \dfrac{%
w+1}{v}\right) ^{2}\right) $, which implies that%
\begin{gather}
(x_{2}-x_{1})^{2}+(y_{2}-y_{1})^{2}=  \label{x1x2} \\
\left( 1+\left( \dfrac{w+1}{v}\right) ^{2}\right) \dfrac{v^{2}s(1-r_{1})%
\zeta }{\beta ^{2}}\text{.}  \notag
\end{gather}

Similarly, suppose that $CD_{2}$\ intersects $E_{I}$ at the two distinct
points $P_{3}=(x_{3},y_{3})=\left( x_{3},\dfrac{w-1}{v}x_{3}+\dfrac{s}{\beta 
}\right) $ and $P_{4}=(x_{4},y_{4})=\left( x_{4},\dfrac{w-1}{v}x_{4}+\dfrac{s%
}{\beta }\right) $. Since $P_{3}$ and $P_{4}$ lie on $E_{I}$, we have $%
A(r_{1})x_{j}^{2}+B(r_{1})x_{j}y_{j}+C(r_{1})y_{j}^{2}+D(r_{1})x_{j}+E(r_{1})y_{j}+I(r_{1})=0,j=3,4 
$. Substituting $y_{j}=\dfrac{w-1}{v}x_{j}+\dfrac{s}{\beta }$ and
simplifying yields 
\begin{gather}
{\large (}A(r_{1})+\dfrac{w-1}{v}B(r_{1})+\left( \dfrac{w-1}{v}\right)
^{2}C(r_{1}){\large )}x_{j}^{2}+  \label{quad2} \\
\left( \dfrac{s}{\beta }B(r_{1})+\dfrac{2s}{\beta }\dfrac{w-1}{v}%
C(r_{1})+D(r_{1})+\dfrac{w-1}{v}E(r_{1})\right) x_{j}+  \notag \\
\left( \dfrac{s}{\beta }\right) ^{2}C(r_{1})+\left( \dfrac{s}{\beta }\right)
E(r_{1})+I(r_{1})=0,j=3,4\text{.}  \notag
\end{gather}

For the quadratic in the variable $x_{j}$ given in (\ref{quad2}) we have 
\begin{eqnarray*}
\tau &=&\left( \dfrac{s}{\beta }B(r_{1})+\dfrac{2s}{\beta }\dfrac{w-1}{v}%
C(r_{1})+D(r_{1})+\dfrac{w-1}{v}E(r_{1})\right) ^{2} \\
&&-4{\large (}A(r_{1})+\dfrac{w-1}{v}B(r_{1})+ \\
&&\left( \dfrac{w-1}{v}\right) ^{2}C(r_{1}){\large )}\left( \left( \dfrac{s}{%
\beta }\right) ^{2}C(r_{1})+\left( \dfrac{s}{\beta }\right)
E(r_{1})+I(r_{1})\right) \text{.}
\end{eqnarray*}

The leading coefficient is $a=A(r_{1})+\dfrac{w-1}{v}B(r_{1})+\left( \dfrac{%
w-1}{v}\right) ^{2}C(r_{1})$. Applying Proposition \ref{P2} again and
simplfying yields 
\begin{eqnarray*}
\tau &=&\dfrac{16r_{1}s^{3}v^{2}\left( r_{1}-1\right) ^{2}\zeta ^{2}}{\beta }%
, \\
a &=&\allowbreak 4s\left( 1-r_{1}\right) \zeta \text{.}
\end{eqnarray*}

Thus $(x_{4}-x_{3})^{2}=\dfrac{\tau }{a^{2}}=\dfrac{16r_{1}s^{3}v^{2}\left(
r_{1}-1\right) ^{2}\zeta ^{2}}{\beta {\large (}\allowbreak 4s\left(
1-r_{1}\right) \zeta {\large )}^{2}}=\dfrac{r_{1}sv^{2}}{\beta }$. Now $%
(y_{4}-y_{3})^{2}=\left( \dfrac{w-1}{v}\right)
^{2}(x_{4}-x_{3})^{2}\Rightarrow (x_{4}-x_{3})^{2}+(y_{4}-y_{3})^{2}=\left(
1+\left( \dfrac{w-1}{v}\right) ^{2}\right) (x_{4}-x_{3})^{2}$, which implies
that

\begin{gather}
(x_{4}-x_{3})^{2}+(y_{4}-y_{3})^{2}=  \label{x3x4} \\
\left( 1+\left( \dfrac{w-1}{v}\right) ^{2}\right) \dfrac{r_{1}sv^{2}}{\beta }%
\text{.}  \notag
\end{gather}

$L_{1}$ and $L_{2}$ are \textbf{equal }conjugate diameters if and only if $%
\left\vert P_{1}P_{2}\right\vert =\left\vert P_{3}P_{4}\right\vert \iff
(x_{2}-x_{1})^{2}+(y_{2}-y_{1})^{2}=(x_{4}-x_{3})^{2}+(y_{4}-y_{3})^{2}$.
Using (\ref{x1x2}) and (\ref{x3x4}),$\left\vert P_{1}P_{2}\right\vert
=\left\vert P_{3}P_{4}\right\vert \iff \left( 1+\left( \dfrac{w+1}{v}\right)
^{2}\right) \dfrac{v^{2}s(1-r_{1})\zeta }{\beta ^{2}}=\left( 1+\left( \dfrac{%
w-1}{v}\right) ^{2}\right) \dfrac{r_{1}sv^{2}}{\beta }\iff $

$\left( 1+\left( \dfrac{w+1}{v}\right) ^{2}\right) {\large (}%
v^{2}s(1-r_{1})\zeta {\large )}\beta {\large -}\left( 1+\left( \dfrac{w-1}{v}%
\right) ^{2}\right) r_{1}sv^{2}\beta ^{2}=0\iff $

$\left( 1+\left( \dfrac{w+1}{v}\right) ^{2}\right) {\large (}(1-r_{1})\zeta 
{\large )-}\left( 1+\left( \dfrac{w-1}{v}\right) ^{2}\right) r_{1}\beta
=0\iff $

$\dfrac{2\left( s-v\right) \left( v^{2}+w^{2}+1\right) r_{1}^{2}+2v\left(
v^{2}+w^{2}+1\right) r_{1}-s{\large (}v^{2}+(w+1)^{2}{\large )}}{v^{2}}%
=0\iff \alpha (r_{1})=0$, where $\alpha $ is given by (\ref{q}). Since $%
r_{1} $ is a root of $\alpha $ by definition, that completes the proof of
Theorem \ref{T3}(ii).
\end{proof}

\textbf{Example}

Consider the quadrilateral, $Q$, with vertices $(0,0),(0,1),(8,4)$, and $%
(6,2)$. Thus $Q=Q_{s,t,v,w}$ with $s=8$, $t=4$, $v=6$, and $w=2$, which
satisfy (\ref{R0}), (\ref{R1}), and (\ref{R2}). $Q$ is a type 1 midpoint
diagonal quadrilateral since $vt=(w+1)s$. The diagonal lines of $Q$ are $%
\overleftrightarrow{D_{1}}$: $y=\dfrac{1}{2}x$, which is also the equation
of $L$, and $\overleftrightarrow{D_{2}}$: $y=1+\dfrac{1}{6}x$.

First, let $E_{0}$ be the ellipse inscribed\textbf{\ }in $Q$ corresponding
to $r=\dfrac{3}{7}$. By Proposition \ref{P2}(i) the equation of $E_{0}$ is $%
33x^{2}-148xy+196y^{2}+28x-168y=-36$. By (\ref{center}), $E_{0}$ has center $%
\left( \dfrac{7}{2},\dfrac{7}{4}\right) $ and the points of tangency of $%
E_{0}$ with $Q$ are given by $q_{1}=\allowbreak \left( 0,\dfrac{3}{7}\right) 
$, $q_{2}=\allowbreak \left( \dfrac{32}{9},\dfrac{7}{3}\right) $, $%
q_{3}=\allowbreak \left( \dfrac{62}{9},\dfrac{26}{9}\right) $, and $%
q_{4}=\allowbreak \left( \dfrac{18}{7},\dfrac{6}{7}\right) $.

\textbullet\ As guaranteed by Theorem \ref{T2}, the slope of $%
\overleftrightarrow{q_{2}q_{3}}=$ slope of $\overleftrightarrow{q_{1}q_{4}}=%
\dfrac{1}{6}=$ slope of $D_{2}$.

\textbullet\ Suppose that $DI_{1}$ and $DI_{2}$\ are diameters of $E_{0}$
which are parallel to the diagonals $D_{1}$ and $D_{2}$, respectively, and
suppose that $DI_{1}$ intersects $E_{0}$ at the two distinct points $P_{1}$\
and $P_{2}$, while $DI_{2}$ intersects $E_{0}$ at the two distinct points $%
P_{3}$ and $P_{4}$. The equations of $\overleftrightarrow{P_{1}P_{2}}$ and $%
\overleftrightarrow{P_{3}P_{4}}$\ are thus $y-\dfrac{7}{4}=\dfrac{1}{2}%
\left( x-\dfrac{7}{2}\right) $ and $y-\dfrac{7}{4}=\dfrac{1}{6}\left( x-%
\dfrac{7}{2}\right) $, respectively. To determine the coordinates of $P_{1}$%
\ and $P_{2}$, substitute the equation of $DI_{1}$ into the equation of $%
E_{0}$. That yields $33x^{2}-148x\left( \dfrac{7}{4}+\dfrac{1}{2}\left( x-%
\dfrac{7}{2}\right) \right) +196\left( \dfrac{7}{4}+\dfrac{1}{2}\left( x-%
\dfrac{7}{2}\right) \right) ^{2}+28x-168\left( \dfrac{7}{4}+\dfrac{1}{2}%
\left( x-\dfrac{7}{2}\right) \right) =-36$, which has solutions $x=\dfrac{7}{%
2}\pm \dfrac{1}{2}\sqrt{31}$. Hence $P_{1}=\dfrac{1}{4}(7-\sqrt{31})(2,1)$
and $P_{2}=\dfrac{1}{4}(7+\sqrt{31})(2,1)$. Similarly, $P_{3}=\dfrac{1}{4}%
(2,1)\left( 7-3\sqrt{2},7-\sqrt{2}\right) $ and $P_{4}=\dfrac{1}{4}%
(2,1)\left( 7+3\sqrt{2},7+\sqrt{2}\right) $. The family of lines which are
parallel to $\overleftrightarrow{P_{1}P_{2}}$ have equation $y=\dfrac{1}{2}%
x+b$. Substituting into the equation of $E_{0}$ yields $33x^{2}-148x\left( 
\dfrac{1}{2}x+b\right) +196\left( \dfrac{1}{2}x+b\right) ^{2}+28x-168\left( 
\dfrac{1}{2}x+b\right) =-36$. Solving for $x$ gives $x=\dfrac{7}{2}-3b\pm 
\dfrac{1}{2}\sqrt{31-62b^{2}}$, and thus the chords of $E_{0}$ which are
parallel to $\overleftrightarrow{P_{1}P_{2}}$ have midpoints $\left( \dfrac{7%
}{2}-3b,\dfrac{7}{4}-\dfrac{1}{2}b\right) $, which each lie on $%
\overleftrightarrow{P_{3}P_{4}}$ . Hence $DI_{1}$ and $DI_{2}$\ are the pair
of conjugate diameters of $E_{0}$ which are parallel to the diagonals of $Q$%
, as guaranteed by Theorem \ref{T1}.

Second, let $E_{I}$ be the unique ellipse of minimal eccentricity inscribed%
\textbf{\ }in $Q$. \ By (\ref{q}), $\alpha (r)=\allowbreak 164r^{2}+492r-360$
and the unique root of $\alpha $ in $J$ is $r_{1}=-\dfrac{3}{2}+\dfrac{27}{82%
}\sqrt{41}$. By Proposition \ref{P2}(i), after some simplification, the
equation of $E_{I}$ is $3(427-63\sqrt{41})x^{2}+8(-793+117\sqrt{41}%
)xy+164(61-9\sqrt{41})y^{2}+4(-2911+459\sqrt{41})x+24(2911-459\sqrt{41}%
)y=36(\allowbreak -3521+549\sqrt{41})$. By Theorem \ref{T1}, there are
conjugate diameters $DI_{1}$ and $DI_{2}$ of $E_{I}$ which are parallel to
the diagonals $D_{1}$ and $D_{2}$, respectively. Suppose that $DI_{1}$
intersects $E_{I}$ at the two distinct points $P_{1}$\ and $P_{2}$, while $%
DI_{2}$ intersects $E_{I}$ at the two distinct points $P_{3}$ and $P_{4}$.
As we did above, one can show that $P_{1}=\tfrac{-82+18\sqrt{41}-9\sqrt{74}+%
\sqrt{3034}}{20}{\large (}2,1{\large )},P_{2}=\tfrac{-82+18\sqrt{41}+9\sqrt{%
74}-\sqrt{3034}}{20}{\large (}2,1{\large ),}P_{3}=\left( \tfrac{-82+18\sqrt{%
41}-27\sqrt{10}+3\sqrt{\allowbreak 410}}{10},\tfrac{-82+18\sqrt{41}-9\sqrt{10%
}+\sqrt{\allowbreak 410}}{20}\right) \allowbreak $, and

$P_{4}=\left( \tfrac{-82+18\sqrt{41}+27\sqrt{10}-3\sqrt{\allowbreak 410}}{10}%
,\tfrac{-82+18\sqrt{41}+9\sqrt{10}-\sqrt{\allowbreak 410}}{20}\right)
\allowbreak $. It then follows that $\left\vert \overline{P_{1}P_{2}}%
\right\vert ^{2}=\left\vert \overline{P_{3}P_{4}}\right\vert ^{2}=\dfrac{37}{%
5}(61-9\sqrt{41})$. As guaranteed by Theorem \ref{T3}, $\overline{P_{1}P_{2}}
$ and $\overline{P_{3}P_{4}}$ are \textbf{equal} conjugate diameters of $%
E_{I}$ which are parallel to the diagonals of $Q$.

\end{document}